\documentclass[11pt]{amsart}
\usepackage[foot]{amsaddr}
\usepackage{enumerate}
\usepackage{caption,subcaption}

\usepackage{ifthen}

\usepackage{amsmath,amsthm,amssymb,mathtools,amsfonts}
\usepackage{pgf}
\usepackage{graphicx}
\usepackage{tikz}
\usetikzlibrary{backgrounds, patterns, shapes, calc, decorations.pathreplacing, decorations.pathmorphing}

\tikzset{%
    vtx/.style={draw,circle,very thin,white,fill=black,inner sep=0pt,text width=8pt},
    pth/.style={draw,decorate,decoration={snake},thick},
    edg/.style={ultra thick},
    thinedg/.style={thick},
    dashedg/.style={ultra thick, dashed},    
    smedg/.style={very thin},
    polyg/.style={very thin,fill=white},
    dotedg/.style={ultra thick,dotted},
    ellip/.style={very thin,black,fill=white},
    spedg/.style={rounded corners,line width=5pt,purple!50}, 
}
\tikzset{/pgf/foreach/parse=true}
\pgfdeclarelayer{backback}
\pgfdeclarelayer{back}
\pgfdeclarelayer{fore}
\pgfdeclarelayer{forefore}
\pgfsetlayers{backback,back,main,fore,forefore}

\newcommand{\ovaldiagram}{%
    \useasboundingbox (-3.0,0) rectangle (3.0,5.5);
    \def\bigrad{1.2};
    \def\smrad{0.7};
    \coordinate (u) at (-2.0,0.5);
    \coordinate (v) at (2.0,0.5);
    \coordinate (Wuv) at (0,2.0);
    \coordinate (Wu) at (-1.5,4.0);
    \coordinate (Wv) at (1.5,4.0);
    \coordinate (WuvL) at ($(Wuv)-(0.9*\bigrad,0)$);
    \coordinate (WuvR) at ($(Wuv)+(0.9*\bigrad,0)$);
    \coordinate (WuL) at ($(Wu)-(0.9*\smrad,0)$);
    \coordinate (WuR) at ($(Wu)+(0.9*\smrad,0)$);
    \coordinate (WvL) at ($(Wv)-(0.9*\smrad,0)$);
    \coordinate (WvR) at ($(Wv)+(0.9*\smrad,0)$);
    \begin{pgfonlayer}{fore}
        \draw (u) node[vtx,label={left:$u$}] {};
	\draw (v) node[vtx,label={right:$v$}] {};
        \draw[ellip] (Wuv) ellipse ({\bigrad} and {\smrad});
        \draw[ellip,rotate=60] (Wu) ellipse ({\bigrad} and {\smrad});
        \draw[ellip,rotate=-60] (Wv) ellipse ({\bigrad} and {\smrad});
        \draw (Wu) node[xshift=-7pt,yshift=20pt,label={left:$W_u$}] {};
        \draw (Wv) node[xshift=7pt,yshift=20pt,label={right:$W_v$}] {};
        \draw (Wuv) node[yshift=-15pt,label={below:$W_{uv}$}] {};
    \end{pgfonlayer}
    \begin{pgfonlayer}{main}
        \draw[edg] (u) -- (v);
    \end{pgfonlayer}	
    \begin{pgfonlayer}{back}
        \draw[polyg] (u) -- (WuvL) -- (WuvR) -- (u);
        \draw[polyg] (v) -- (WuvL) -- (WuvR) -- (v);
        \draw[polyg] (u) -- (WuL) -- (WuR) -- (u);
        \draw[polyg] (v) -- (WvL) -- (WvR) -- (v);
    \end{pgfonlayer}
}

\newtheorem{theorem}{Theorem}
\newtheorem{lemma}[theorem]{Lemma}

\newtheorem{definition}[theorem]{Definition}
\newtheorem{fact}[theorem]{Fact}
\newtheorem{observation}[theorem]{Observation}
\newtheorem{conjecture}[theorem]{Conjecture}
\newtheorem{question}[theorem]{Question}
\newtheorem*{remark}{Remark}

\newcommand{\C}{{\mathcal C}}
\newcommand{\K}{{\mathcal K}}

\keywords{graph colorings, planar graphs}
\subjclass[2020]{05C15,05C10,05C35}

\begin{document}
\title{Bounds for B-coloring planar and outerplanar graphs}
\author[R.R. Martin]{Ryan R. Martin$^1$}
\address[1]{Department of Mathematics, Iowa State University, Ames, Iowa, USA}
\email[1]{rymartin@iastate.edu}
\author[M. Ruszink\'o]{Mikl\'os Ruszink\'o$^{2,3,4}$}
\address[2]{HUN‐REN Alfr\'ed R\'enyi Institute of Mathematics, Budapest, Hungary}
\address[3]{Department of Sciences and Engineering, Sorbonne University Abu Dhabi, SAFIR, CEFAMSDS, Abu Dhabi, UAE}
\address[4]{Faculty of Information Technology and Bionics, P\'azm\'any P\'eter Catholic University}
\author[G.N. S\'ark\"ozy]{G\'{a}bor N. S\'ark\"ozy$^{2,5}$}
\address[5]{Computer Science Department, Worcester Polytechnic Institute, Worcester, MA, USA} 

\renewcommand{\theenumi}{\alph{enumi}}

\begin{abstract}
    A proper coloring of the edges of a graph is one in which every $K_{1,2}$ is totally multicolored. 
    A B-coloring is a coloring of the edges of $G$ in which every $K_{1,2}$ and every $K_{2,2}$ is totally multicolored.

    In this paper, we improve two results by the authors and Gy\'arf\'as. The first shows that an outerplanar graph with maximum degree $\Delta$ can be B-colored with $\max\{\Delta,9\}$ colors, improving a bound of $\Delta+1$ and this is best possible up to the constant 9.

    The second shows that a planar graph with maximum degree $\Delta$ can be B-colored with $\max\{2\Delta+4,22\}$, improving a bound of $2\Delta+8$. 
\end{abstract}

\maketitle

\section{Introduction}

A {\em planar graph} is a graph that can be drawn on the Euclidean plane such that any two edges can only intersect at their endvertices.
An {\em outerplanar graph} is a planar graph that can be drawn such that all of its vertices lie on the boundary of the same face.
In this paper, all of our graphs will be simple.
That is, they will have no loops or multiple edges. 

An edge coloring of a graph $G$ is {\em proper} if adjacent edges must receive different colors. 
In this case, the edges incident with the same vertex are said to be {\em totally multicolored} or {\em rainbow}. 
Equivalently, the edges of any copy of the complete bipartite graph $K_{1,2}$ must be different. 
Vizing's Theorem~\cite{V} says that any graph $G$ with maximum degree $\Delta$ can be properly colored with either $\Delta$ or $\Delta+1$ colors. 
Holyer~\cite{H} proved that it is NP-hard to determine whether $\Delta$ colors suffice for an arbitrary graph $G$. 

The $(7,4)$-problem of Brown, Erd\H{o}s and S\'os~\cite{BETS} claims that any triple system (i.e., $3$-uniform hypergraph) on $n$ vertices such that no set of $7$ vertices has $4$ triples has a total of $o(n^2)$ triples. 

Closely related to the (7,4)-problem, Gy\'arf\'as and S\'ark\"ozy~\cite{GYS1} introduced and studied {\em A-} and {\em B-colorings}: an A-coloring is a proper coloring where the union of any two color classes does not contain paths or cycles with four edges and a B-coloring is a proper coloring where  every 4-cycle of $G$ must be rainbow. 
A-colorings are also called {\em star-edge colorings} in the literature and were defined in \cite{DMS} (they appeared also in \cite{EGY}) and they have been extensively studied.
See, for example,~\cite{BLMSS,DYZC,WWW}, or the survey~\cite{LS}.
We define $q_A(G)$ [$q_B(G)$] to be the minimum number of colors sufficient for an A-coloring [B-coloring] of the graph $G$.

In~\cite{GYMRS}, $q_B(G)$ was studied for planar graphs. 
After proving some partial results, Gy\'arf\'as and the present authors formulated the following conjecture.
\begin{conjecture}[\cite{GYMRS}] \label{conj:GYMRS}
    There is a $\Delta_0$ such that for $\Delta \geq \Delta_0$ we have $q_B(G) \leq 2 \Delta$ for any planar $G$ with maximum degree $\Delta(G)=\Delta$.  
    Similarly, for large enough $\Delta$, 
    $q_B(G) \leq  \Delta$ for any outerplanar $G$ with maximum degree $\Delta$.
\end{conjecture}

For planar graphs, the fact that $2\Delta$ colors may be necessary is shown by $K_{2,\Delta}$ since all edges must get a different color in any B-coloring.
For outerplanar graphs, it is immediate from Vizing's theorem.
However, the constant $\Delta_0$ is needed.
Examples can be seen in Figure 1.

We note that for A-colorings the corresponding problem is still wide open (see \cite{LS}).

The purpose of this paper is to prove the conjecture for outerplanar graphs and to give an improved bound for planar graphs.
Instead of expressing our results in the manner of Conjecture~\ref{conj:GYMRS}, we instead express it so that, for some function $f$ and constant $c$,  if the number of colors in the palette is at least $\max\bigl\{f(\Delta),c\bigr\}$ then a B-coloring of the graph with that palette is possible.

First for outerplanar graphs we prove the following, which improves on the bound of $\Delta+1$ from~\cite{GYMRS}:
\begin{theorem} \label{thm:outerplanar}
	Let $G$ be an outerplanar graph with maximum degree at most $\Delta$.
	Then $q_B(G) \leq \max\{\Delta,9\}$. 
\end{theorem}

In Figure~\ref{fig:example:outerplanar}, there is an outerplanar graph with $\Delta=4$ that requires $5$ colors, so it is an open question as to whether the expression $\max\{\Delta,9\}$ can be replaced by $\max\{\Delta,5\}$.

For planar graphs our result is the following, which improves on the bound of $2\Delta+8$ from~\cite{GYMRS}:
\begin{theorem} \label{thm:planar}
	Let $G$ be a planar graph with maximum degree at most $\Delta$.
	Then $q_B(G) \leq \max\{2\Delta+4,22\}$.
\end{theorem}

This is still short of the conjectured value of $\max\{2\Delta,c\}$ for some constant $c$.
In Figure~\ref{fig:example:planar}, there is a planar graph with $\Delta=4$ that requires $12$ colors, so it is open as to whether the value of $22$ can be reduced to $12$. 

\begin{figure}
	\begin{subfigure}[b]{0.48\textwidth}
		\centering
        \begin{tikzpicture}[scale=1]
            \centering
			\useasboundingbox (-1.2,-1.2) rectangle (1.2,1.2);
			\coordinate (a) at (90:1);
			\coordinate (b) at (18:1);		
			\coordinate (c) at (-54:1);
			\coordinate (d) at (-126:1);
			\coordinate (e) at (162:1);
			\begin{pgfonlayer}{forefore}
				\draw (a) node[vtx] {};
				\draw (b) node[vtx] {};
				\draw (c) node[vtx] {};
				\draw (d) node[vtx] {};
				\draw (e) node[vtx] {};
			\end{pgfonlayer}
			\begin{pgfonlayer}{fore}
				\draw[edg] (a) -- (b);
				\draw[edg] (b) -- (c);
				\draw[edg] (c) -- (d);
				\draw[edg] (d) -- (e);
				\draw[edg] (e) -- (a);
				\draw[edg] (a) -- (c);
				\draw[edg] (a) -- (d);
			\end{pgfonlayer}
		\end{tikzpicture}
		\caption{Outerplanar graph with $\Delta=4$ but $\chi_B=5$.}
        \label{fig:example:outerplanar}
	\end{subfigure} \hfill
	\begin{subfigure}[b]{0.48\textwidth}
        \centering
		\begin{tikzpicture}[scale=1]
			\useasboundingbox (-1.2,-1.2) rectangle (1.2,1.2);
			\coordinate (a) at (0,0.8);
			\coordinate (b) at (-1,0);		
			\coordinate (c) at (-0.33,0);
			\coordinate (d) at (0.33,0);
			\coordinate (e) at (1,0);
	                \coordinate (f) at (0,-0.8);
			\begin{pgfonlayer}{forefore}
				\draw (a) node[vtx] {};
				\draw (b) node[vtx] {};
				\draw (c) node[vtx] {};
				\draw (d) node[vtx] {};
				\draw (e) node[vtx] {};
				\draw (f) node[vtx] {};
			\end{pgfonlayer}
			\begin{pgfonlayer}{fore}
				\draw[edg] (a) -- (b) (b) -- (f);
				\draw[edg] (a) -- (c) (c) -- (f);
				\draw[edg] (a) -- (d) (d) -- (f);
				\draw[edg] (a) -- (e) (e) -- (f);
				\draw[edg] (b) -- (c) (c) -- (d) (d) -- (e);
                    		\draw[edg] (b) to[out=90,in=90,looseness=2] (e);
			\end{pgfonlayer}
		\end{tikzpicture}
		\caption{Planar graph with $\Delta=4$ but $\chi_B=12$.}
        \label{fig:example:planar}
	\end{subfigure}
    \label{fig:example}
	\caption{Small examples}
\end{figure}

A graph $G$ is said to be a {\em topological copy} of a graph $H$ if $G$ can be formed by subdivisions of the edges of $H$.
If $G$ is a graph with a partial B-coloring of the edges and the edge $uv$ is uncolored, we say that an edge that is incident to $u$, incident to $v$ or forms a $C_4$ with $uv$ and receives some color is a {\em conflict edge}. 
A color of a conflict edge of $uv$ is a {\em conflict color}.
We say that two edges are {\em parallel} if they are in the same $C_4$ but do not share a vertex. 

\section{Proof of Theorem~\ref{thm:outerplanar}, the outerplanar case.}

We take note of the following well-known observations of outerplanar graphs:
\begin{lemma}
    The following properties are true of outerplanar graphs:
    \begin{itemize}
        \item The property of outerplanarity is hereditary. That is, it is closed under vertex deletion.
        \item The property of outerplanarity is minor-closed. That is, it is closed under edge-contraction.
    \end{itemize}
	Let $G$ be an outerplanar graph on $n$ vertices.
	\begin{enumerate}
		\item 	$G$ has no topological copy of $K_4$. \label{outerplanar:noK4}
		\item 	$G$ has no topological copy of $K_{2,3}$. \label{outerplanar:noK23}
		\item 	If $n\geq 2$, then $e(G)\leq 2n-3$. \label{outerplanar:2n}
            	\item   	$G$ has at least two vertices of degree at most two. \label{outerplaner:2deg2}
		\item 	If $G$ is $2$-connected, then $G$ is hamiltonian.  \label{outerplanar:hamiltonian}
            \item   If $G$ is $2$-connected, then either $G$ is a cycle or there is a vertex of degree $4$ that has two neighbors with degree $2$ or a vertex of degree $3$ that has at least one neighbor with degree $2$. \label{outerplanar:deg4deg2}
	\end{enumerate}
    \label{lem:structure:outerplanar}
\end{lemma}

Note that if we contract an edge of a planar graph and the edge is in a triangle, then the graph is no longer be simple. 
Thus we will take care to ensure that any edge we contract will not be in a triangle. 

\begin{proof}
    We will only prove~\eqref{outerplanar:deg4deg2}, the others can be found e.g. in West's textbook~\cite{DW}.
    Let $G$ be a 2-connected outerplanar graph.
    Let $V_2\subseteq V=V(G)$ be the set of degree-2 vertices.
    Observe that for any Hamilton cycle in $G$, both edges incident to a degree-$2$ vertex are edges in that Hamilton cycle.
    If $G$ is not a cycle, then by the fact that outerplanarity is a hereditary property, the graph $G'=G[V-V_2]$ is an outerplanar graph with at least one edge.
    (Note that $G'$ might not be connected, let alone $2$-connected.)

    Since $G$ is hamiltonian and every edge incident to a degree-$2$ vertex belongs to every Hamilton cycle, no vertex is incident to more than two vertices in $V_2$.
    Since $G'$ is outerplanar and not empty, there is a vertex of degree at most 2, but that means that, in $G$, it either has degree 4 and has two neighbors in $V_2$ or it has degree 3 and has one or two neighbors in $V_2$.
\end{proof}

Lemma~\ref{lem:main:outerplanar} below is the main technical lemma for  Theorem~\ref{thm:outerplanar} which gives a bound on the number of conflict edges. 
The proof is in Section~\ref{sec:main:outerplanar}.
\begin{lemma}\label{lem:main:outerplanar}
	In an outerplanar graph $G$, the number of edges that conflict with the edge $xy$ is at most $d(x)+d(y)+\min\{d(x),d(y),5\}-3$.

    If there exists a degree-$2$ vertex $z$ such that $xyz$ is a triangle, then the number of edges that conflict with the edge $xy$ is at most $d(x)+d(y)+\min\{d(x),d(y),4\}-4$.

    If the edge $xy$ is not in a $C_4$, then the number of edges that conflict with the edge $xy$ is at most $d(x)+d(y)-2$.
\end{lemma}

With these important lemmas, we proceed to the proof of the outerplanar result. 

\begin{proof}[Proof of Theorem~\ref{thm:outerplanar}]
    With foresight, we will set $c=9$ but use $c$ for convenience and to prevent confusion. 

	The proof will proceed by induction on $|E(G)|$ but we must be careful.
    If an edge $e$ is deleted, then a $C_4$ may be destroyed, which would remove a constraint between parallel edges of this $C_4$. 
    Most of the time, we will delete vertices to avoid this complication. 
	The case of $|E(G)|\leq\max\{\Delta,c\}$ is trivial because the edges can be totally multicolored.

    Now suppose that any outerplanar graph on fewer than $|E(G)|$ edges with maximum degree at most $\Delta$ can be B-colored with at most $\max\{\Delta,c\}$ colors.

    First, we may assume that $G$ is $2$-connected.
    Otherwise if $G$ is the result of two graphs $G_1$ and $G_2$ that share only a common vertex $v$, then we can B-color each independently with $\max\{\Delta,c\}$ colors and permute the colors on $G_2$ so that the edges at $v$ all receive different colors. 
    This results in a B-coloring of $G$ because $v$ cannot be in a $C_4$ that uses vertices from both $G_1-v$ and $G_2-v$. 

    Fact~\ref{fact:outer:conflict} below results from the fact that if there is a vertex $u$ such that no edges incident to it have at least $\max\{\Delta,c\}$ conflict edges, then we can delete $u$, color $G-u$ with $\max\{\Delta,c\}$ colors and then greedily color the edges incident to $u$ because there will always be a color remaining.

    \begin{fact}\label{fact:outer:conflict}
        We may assume that every vertex $u$ is adjacent to a vertex $v$ such that $uv$ conflicts with at least $\max\{\Delta,c\}$ other edges. 
    \end{fact}

    Fact~\ref{fact:outer:uv} gives us conditions on which we cannot apply induction. 
    For part \eqref{it:outer:C4}, we would like to delete an edge $uv$ not in a $C_4$, B-color the edges of $G-uv$, and then color $uv$ if a color remains that does not conflict with the previously-colored edges. 
    This can be done unless the sum $d(u)+d(v)$ is too large. 
    For part \eqref{it:outer:C3}, we would like to contract the edge $uv$, B-color the edges of the graph that results from contracting the edge $uv$ (which we denote with $G/uv$), and then color $uv$ if a color remains which does not conflict with the colors on the previously-colored edges. 
    Part \eqref{it:outer:C3deg2} is a special case of \eqref{it:outer:C3}, which provides a better bound when one vertex has degree 2.

    \begin{fact}\label{fact:outer:uv}
        Let $u,v\in V(G)$ and $uv\in E(G)$. 
        We may assume that each of the following occur:
        \begin{enumerate}[(a)]
            \item If $uv$ is not in a $C_4$, then $d(u)+d(v)\geq \max\{\Delta,c\}+2$. \label{it:outer:C4}
            \item If $uv$ is not in a $C_3$, then either $d(u)+d(v)\geq\Delta+3$ or $d(u)+d(v)\geq \max\{\Delta,c\}-2$. \label{it:outer:C3}
            \item If $uv$ is not in a $C_3$ and $d(u)=2$, then  $d(v)\geq\max\{\Delta,c\}-1$. \label{it:outer:C3deg2}
            \item Every degree-$2$ vertex has a neighbor of degree at least $\max\{\Delta,c\}-1$. \label{it:outer:deg}
        \end{enumerate}
    \end{fact}~\\

    \begin{proof}[Proof of Fact~\ref{fact:outer:uv}]~\\
        \begin{itemize}
            \item [(\ref{it:outer:C4})] By Lemma~\ref{lem:main:outerplanar}, the number of edges that conflict with $uv$ is at most $d(u)+d(v)-2$.
            By induction, we may B-color $G-\{uv\}$ by $\max\{\Delta,c\}$ colors and as long as $d(u)+d(v)-2<\max\{\Delta,c\}$, there is a color available for $uv$. 

            \item [(\ref{it:outer:C3})] By contracting $uv$, the fact that outerplanarity is a minor-closed property gives that the resulting graph $G'=G/uv$ is outerplanar and has maximum degree $\max\{\Delta,d(u)+d(v)-2\}$. 
            Since $uv$ was not in a triangle, $G'$ is a simple graph.
            Moreover, every pair of edges in $G'$ are that are in conflict in $G$ are also in conflict in $G'$:
            If they are adjacent edges in $G$, they remain adjacent in $G'$. 
            If they are parallel edges in $G$, then either they remain parallel in $G'$ or they are in a $C_4$ which contains $uv$ as an edge in which case, the edges become adjacent in $G'$. 
            (We note that there could be more constraints in $G'$ than in $G$, say there could be two edges in some induced $C_5$ that includes $uv$ that become parallel edges in $G'$.)
        
            If $d(u)+d(v)-2\leq\Delta$, then $G'$ has maximum degree at most $\Delta$, and every pair of edges in $E(G)-\{uv\}$ that conflict in $G$ will also conflict in $G'$, we may apply induction to $G'$.
            By Lemma~\ref{lem:main:outerplanar}, there are at most $d(u)+d(v)+5-3$ edges that conflict with $uv$ and if $d(u)+d(v)+2<\max\{\Delta,c\}$, then there is a color available for $uv$. 
            
            \item [(\ref{it:outer:C3deg2})] We follow the same reasoning as in \eqref{it:outer:C3} but observe first that if $d(u)=2$, then $d(u)+d(v)-2\leq\Delta$. 
            Again, we apply induction to $G'$.
            Then note that since $\min\bigl\{d(u),d(v),5\bigr\}=2$, Lemma~\ref{lem:main:outerplanar} gives that there are at most $d(u)+d(v)+2-3=d(v)+1$ edges that conflict with $uv$ and if $d(v)+1<\max\{\Delta,c\}$, then there is a color available for $uv$. 
            
            \item [(\ref{it:outer:deg})] By Lemma~\ref{lem:main:outerplanar}, if $d(u)=2$ then for any neighbor $v$, the number of edges that conflict with $uv$ is at most $d(u)+d(v)+2-3=d(v)+1$. 
            By Fact~\ref{fact:outer:conflict}, there is a neighbor $v$ such that $uv$ conflicts with at least $\max\{\Delta,c\}$ other edges and so there must be a vertex $v$ adjacent to $u$ such that $d(v)+1\geq\max\{\Delta,c\}$. 
        \end{itemize}
    \end{proof}

    With Fact~\ref{fact:outer:uv} in mind, let us prove the theorem. 
    First, observe that since $G$ is $2$-connected, then Lemma~\ref{lem:structure:outerplanar}(\ref{outerplanar:hamiltonian}) gives that $G$ is hamiltonian.
    We call the edges of some fixed Hamilton cycle \textit{outer edges}. 
    
    In addition, by Lemma~\ref{lem:structure:outerplanar}(\ref{outerplanar:deg4deg2}), 
	there must be at least one structure, which we refer to as a gadget:
    \begin{itemize}
        \item[($D_4$)] a degree-$4$ vertex $v$ adjacent to two degree-$2$ vertices $u$ and $u'$, or
        \item[($D_3$)] a degree-$3$ vertex $v$ adjacent to a degree-$2$ vertex $u$.
    \end{itemize}
    Each of the edges $uv$ and $u'v$ (in the case of $D_4$) are outer edges.
    
    Consider a $D_4$ and let $w$ be the neighbor of $u$ where $uw$ is an outer edge and $w'$ be the neighbor of $u'$ where $u'w'$ is an outer edge as illustrated in Figure~\ref{fig:gadget:D4}.
    By Fact~\ref{fact:outer:uv}(\ref{it:outer:C3deg2}), since $d(v)=4<\max\{\Delta,c\}-1$ ($c\geq 6$ suffices), the edge $uv$ must be in a triangle and the only choice is $uvw$. 
    So $uw$ (and similarly $u'w'$) must be an edge.
    Furthermore, by Fact~\ref{fact:outer:uv}(\ref{it:outer:C4}), since $d(u)+d(v)=6<\max\{\Delta,c\}+2$ ($c\geq 5$ suffices), the edge $uv$ must be in a $C_4$.
    This must be $uww'v$.
    Furthermore, by Fact~\ref{fact:outer:uv}(\ref{it:outer:deg}), each of $d(w)$ and $d(w')$ is at least $\max\{\Delta,c\}-1$.

    Now consider $D_3$ and let $w$ be the neighbor of $u$ and $w'$ be the neighbor of $v$ via an outer edge. 
    The vertices $w$ and $w'$ must be distinct, otherwise $G$ cannot be $2$-connected.
    By Fact~\ref{fact:outer:uv}(\ref{it:outer:C3deg2}), since $d(v)=3<\max\{\Delta,c\}-1$, the edge $uv$ must be in a triangle and the only choice is $uvw$. 
    Furthermore, by Fact~\ref{fact:outer:uv}(\ref{it:outer:C4}), since $d(u)+d(v)=5<\max\{\Delta,c\}+2$, the edge $uv$ must be in a $C_4$.
    This must be $ww'$.

    Thus, Figure~\ref{fig:gadget:D4} and Figure~\ref{fig:gadget:D3} show the gadgets $D_4$ and $D_3$, respectively.

    \begin{figure}
        \centering
        \begin{subfigure}[b]{0.35\textwidth}
            \centering
            \begin{tikzpicture}[scale=1]
                \centering
                \useasboundingbox (-1.2,-0.7) rectangle (1.2,2.0);
                \coordinate (u) at (-1,0);
                \coordinate (v) at (0,0);		
                \coordinate (up) at (1,0);
                \coordinate (w) at (-1,1);
                \coordinate (wp) at (1,1);
                \begin{pgfonlayer}{fore}
                    \draw (u) node[vtx] {};
                    \draw (v) node[vtx] {};
                    \draw (up) node[vtx] {};
                    \draw (w) node[vtx] {};
                    \draw (wp) node[vtx] {};
                \end{pgfonlayer}
                \begin{pgfonlayer}{back}
                    \draw (u) node[vtx,label={below,yshift=-4pt:$u$}] {};
                    \draw (v) node[vtx,label={below,yshift=-4pt:$v$}] {};
                    \draw (up) node[vtx,label={below:$u'$}] {};
                    \draw (w) node[vtx,label={above:$w$}] {};
                    \draw (wp) node[vtx,label={above:$w'$}] {};
                \end{pgfonlayer}
                \begin{pgfonlayer}{main}
                    \draw[edg] (w) -- (u);
                    \draw[edg] (u) -- (v);
                    \draw[edg] (v) -- (up);
                    \draw[edg] (up) -- (wp);
                    \draw[thinedg] (w) -- (v);
                    \draw[thinedg] (v) -- (wp);
                    \draw[dashedg] (w) -- (wp);
                \end{pgfonlayer}
            \end{tikzpicture}
            \caption{Gadget $D_4$}
         	\label{fig:gadget:D4}
        \end{subfigure} 
        \begin{subfigure}[b]{0.35\textwidth}
            \centering
            \begin{tikzpicture}[scale=1]
                \centering
                \useasboundingbox (-0.2,-0.7) rectangle (1.7,2.0);
                \coordinate (u) at (0,0);
                \coordinate (v) at (1.5,0);
                \coordinate (w) at (0,1);
                \coordinate (wp) at (1.5,1);
                \begin{pgfonlayer}{fore}
                    \draw (u) node[vtx] {};
                    \draw (v) node[vtx] {};
				\draw (w) node[vtx] {};
				\draw (wp) node[vtx] {};
			\end{pgfonlayer}
			\begin{pgfonlayer}{back}
				 \draw (u) node[vtx,label={below,yshift=-4pt:$u$}] {};
				 \draw (v) node[vtx,label={below,yshift=-4pt:$v$}] {};
				 \draw (w) node[vtx,label={above:$w$}] {};
				 \draw (wp) node[vtx,label={above:$w'$}] {};
			\end{pgfonlayer}
			\begin{pgfonlayer}{main}
				\draw[edg] (w) -- (u);
				\draw[edg] (u) -- (v);
				\draw[edg] (v) -- (wp);
				\draw[thinedg] (w) -- (v);
				\draw[thinedg] (v) -- (wp);
				\draw[dashedg] (w) -- (wp);
			\end{pgfonlayer}
		\end{tikzpicture}
		\caption{Gadget $D_3$}
         \label{fig:gadget:D3}
	\end{subfigure} 
	\caption{Gadgets for the proof of Theorem~\ref{thm:outerplanar}. Bold edges are outside edges. Dotted edges belong to $G-\{u,v,u'\}$ or $G-\{u,v\}$, respectively.}
\end{figure}
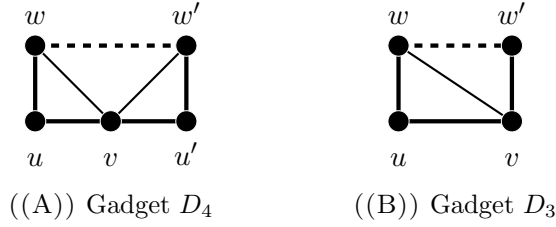

    So now we have a strong characterization of the vertices $w$ and $w'$.

    \begin{fact}\label{fact:outer:D4D3}
        If $G$ is not B-colorable by $\max\{\Delta,c\}$ colors and $c\geq 6$, then $\Delta\geq c$ and the following are true for gadgets:
        \begin{itemize}
            \item[($D_4$)] Both $w$ and $w'$ have degree $\Delta$.
            \item[($D_3$)] The vertex $w$ has degree $\Delta$ and the vertex $w'$ has degree at least $\Delta-2$. 
        \end{itemize}
    \end{fact}

    \begin{proof}[Proof of Fact~\ref{fact:outer:D4D3}]
        First we show that every vertex of degree 2 either has a neighbor of degree $\Delta$ or both neighbors have degree $\Delta-1$.
        Towards a contradiction, consider a vertex $u$ of degree two with neighbors $v_1$ and $v_2$ such that $d(v_1)\geq d(v_2)$ and suppose $d(v_1)\leq\Delta-1$ and $d(v_2)\leq\Delta-2$. 

        By Lemma~\ref{lem:main:outerplanar}, the edge $uv_1$ has at most $d(v_1)+2+2-3=d(v_1)+1$ conflicts, including the edge $uv_2$. The edge $uv_2$ has at most $d(v_2)+1$ conflicts.

        First, B-color the graph $G-u$ with at most $\max\{\Delta,c\}$ colors. 
        Second color $uv_1$. 
        There are at most $d(v_1)\leq\Delta-1\leq\max\{\Delta,c\}-1$ colored edges that conflict with $uv_1$ (since $uv_2$ is uncolored) and so a remaining color is available for $uv_1$.
        Third color $uv_2$. 
        There are at most $d(v_2)+1\leq\max\{\Delta,c\}-1$ colored edges that conflict with $uv_2$ and so a remaining color is also available for $uv_2$.        
        This is a B-coloring of $G$, a contradiction. 
        Therefore, either $d(v_1)\geq\Delta$ or $d(v_1)=d(v_2)=\Delta-1$ and observe that if $c>\Delta$, then there is still a color available in both cases. 

        Turning to $D_4$, the above argument implies that, as long as $4<\max\{\Delta,c\}-1$ ($c\geq 6$ suffices), it must be the case that $d(w)=d(w')=\Delta$. 

        Turning to $D_3$, the above argument also implies that $d(w)=\Delta$.
        To bound $d(w')$, we will suppose $d(w')\leq\Delta-3$.
        First, B-color $G-\{u,v\}$ (the ``$-$'' symbol in this case defines the graph $G$ with vertices $u$ and $v$ deleted) with at most $\max\{\Delta,c\}$ colors and color the following edges in order: $vw$, $uw$, $vw'$ and $uv$.
        We know that $d(u)=2$, $d(v)=3$, and $d(w)=\Delta$.
        Consider Lemma~\ref{lem:main:outerplanar}.
    
        \begin{itemize}
            \item[$vw$:]  The number of edges that conflict with $vw$ is at most $d(v)+d(w)+\min\{d(v),d(w),4\}-4=\Delta+2$, including the three uncolored edges $uw$, $vw'$, $uv$. 
            Thus, there is a color available for $vw$. 
            \item[$uw$:]  The number of edges that conflict with $uw$ is at most $d(u)+d(w)+\min\{d(u),d(w),5\}-3=\Delta+1$, including the two uncolored edges $vw'$, $uv$. 
            Thus, there is a color available for $uw$. 
            \item[$vw'$:]  The number of edges that conflict with $vw'$ is at most $d(v)+d(w')+\min\{d(v),d(w'),5\}-3=d(w')+3\leq\Delta$, including the uncolored edge $uv$. 
            Thus, there is a color available for $uw$. 
            \item[$uv$:]  The number of edges that conflict with $uv$ is at most $d(u)+d(v)+\min\{d(u),d(v),5\}-3=4$. 
            Since $\max\{\Delta,c\}\geq 5$, there is a color available for $uw$.
        \end{itemize}

        This concludes the proof of Fact~\ref{fact:outer:D4D3}.
    \end{proof}

    So now we have that $\Delta\geq c$ and the set of all $w$'s and $w'$'s in a gadget have degree $\Delta$ except those vertices that serve only as a $w'$ in gadgets $D_3$, which have degree at least $\Delta-2$.

    Now, delete the vertices in each of the gadgets, other than those labeled $w$ and $w'$. 
    That is, for a $D_4$, delete $\{u,v,u'\}$ and for a $D_3$ delete $\{u,v\}$.

    The resulting degrees of the vertices that were in gadgets are at least $\Delta-2-2=\Delta-4$ (in the case of $D_4$ or a $w$ vertex in some $D_3$) or at least $(\Delta-2)-1-1=\Delta-4$ (in the case of vertices that are only $w'$ vertices in $D_3$'s). This is because any vertex labeled $w$ or $w'$ can be a part of at most two gadgets and deleting a gadget will delete at most one outer edge from such a vertex. 

    If $\max\{\Delta,c\}>8$ (so $c\geq 9$ suffices), then no new vertices of degree at most $4$ will be created. 
    But this is a contradiction to Lemma~\ref{lem:structure:outerplanar}\eqref{outerplanar:deg4deg2} because no new gadgets will be created but all the existing ones will be deleted.

    Thus, having $\max\{\Delta,9\}$ colors suffices to B-color an outerplanar graph with maximum degree $\Delta$. 
\end{proof}

\section{Proof of Theorem~\ref{thm:planar}, the planar case.}

We take note of the following well-known observations of planar graphs:
\begin{lemma}
    The following properties are true for planar graphs:
    \begin{itemize}
        \item The property of planarity is hereditary. That is, it is closed under vertex deletion.
        \item The property of planarity is minor-closed. That is, it is closed under edge-contraction. 
    \end{itemize}
	Let $G$ be a planar graph on $n$ vertices.
	\begin{enumerate}
		\item 	$G$ has no topological copy of $K_5$.
		\item 	$G$ has no topological copy of $K_{3,3}$.
		\item 	If $n\geq 3$, then $e(G)\leq 3n-6$. Otherwise, $e(G)\leq 3n$. \label{it:planar:edges}
            \item   If $G$ is bipartite and $n\geq 3$, then $e(G)\leq 2n-4$. Otherwise if $G$ is bipartite, then $e(G)\leq 2n$. \label{it:biplanar:edges}
            \item   The minimum degree of $G$ is at most five. \label{it:planar:deg}
	\end{enumerate}
    \label{lem:structure:planar}
\end{lemma}

Lemma~\ref{lem:main:planar} below is the main technical lemma for Theorem~\ref{thm:planar} which gives a bound on the number of conflict edges. 
The proof is in Section~\ref{sec:main:planar}.

\begin{lemma}      \label{lem:main:planar}	
    Let $G$ be a connected planar graph that is not $K_2$ and let $u$ and $v$ be adjacent vertices. 

    If $\min\{\deg(u),\deg(v)\}=1$, then the number of edges that form a conflict with $uv$ is at most $\max\bigl\{\deg(u),\deg(v)\bigr\}-1$.

    Otherwise, the number of edges that form a conflict with $uv$ is at most:
    \begin{equation}\label{bound}
        2\deg(u) + 2\deg(v) + \left\lfloor\frac{1}{2}\min\bigl\{\deg(u),\deg(v)\bigr\}\right\rfloor - 6.
    \end{equation}
\end{lemma}

With these important lemmas, we proceed to the proof of the planar result. 

\begin{proof}[Proof of Theorem~\ref{thm:planar}]
    With foresight, we will set $c=22$ but use $c$ for convenience and to prevent confusion. 

    In this case, we can proceed by induction on $|V(G)|$.
    The base case of $\binom{|V(G)|}{2}\leq \max\{2\Delta,c\}$ is trivial because the edges can be totally multicolored.

    Now we consider the following theorem of Borodin and Ivanova~\cite{BI}:
    \begin{theorem}[Borodin, Ivanova~\cite{BI}]\label{thm:BI}
        Let $G$ be a planar graph with minimum degree $\delta\geq 3$. 
        There is a vertex $u$ of degree $d\geq\delta$ with $d-2$ neighbors $w_1,\ldots,w_{d-2}$ such that one of the following occurs:
        \begin{enumerate}[(i)]
            \item $d=3$ and $\deg(w_1)\leq 10$, \label{it:BI:3}
            \item $d=4$ and $\deg(w_2)\leq 5$ and $\deg(w_1)\leq 9$, \label{it:BI:4a}
            \item $d=4$ and $\deg(w_2)=6$ and $\deg(w_1)\leq 8$, \label{it:BI:4b}
            \item $d=4$ and $\deg(w_2)=\deg(w_1)=7$, \label{it:BI:4c}
            \item $d=5$ and $\deg(w_3)=4$ and $\deg(w_2)=\deg(w_1)=5$, \label{it:BI:5a}
            \item $d=5$ and $\deg(w_3)=5$ and $\deg(w_2)\in\{5,6\}$ and $\deg(w_1)\in\{5,6,7\}$, \label{it:BI:5b}
            \item $\delta=5$ and $\deg(w_3)=\deg(w_2)=\deg(w_1)=6$. \label{it:BI:5c}
        \end{enumerate}
    \end{theorem}

    Using Lemma~\ref{lem:structure:planar}\eqref{it:planar:deg}, we see that there is a vertex $u$ of degree at most $5$.
    If $\delta=\delta(G)\leq 2$, then let $u$ be a vertex of minimum degree $\delta$. 
    Otherwise, let $u$ be the vertex given by Theorem~\ref{thm:BI}.
    Let $d=\deg(u)$.

    If $d\geq 3$, then let $N(u)=\bigl\{v_1,v_2,w_1,\ldots,w_{d-2}\bigr\}$. 
    If $\deg(u)=2$, then let $N(u)=\{v_1,v_2\}$ and if $\deg(u)=1$, then let $N(u)=\{v\}$.
    By induction, $G-u$ can be B-colored by $\max\{2\Delta+4,c\}$ colors. 

    In each case below, the bound we get on the number of conflicting edges is derived from Lemma~\ref{lem:main:planar}. \\

    \noindent\textbf{Case 1. $d=1$.}

    The number of edges that conflict with $uv$ is at most $\deg(v)-1\leq\Delta-1$. 
    So there are at least $\max\{2\Delta+4,c\}-(\Delta-1)\geq\Delta+5\geq 1$ colors available for $uv$. \\
    
    \noindent\textbf{Case 2. $d=2$.}

    The number of edges that conflict with $uv_i$ for $i\in\{1,2\}$ is at most $2\cdot 2+2\deg(v_i)+\lfloor 2/2 \rfloor-6=2\deg(v_i)-1\leq 2\Delta-1$.

    If we color $uv_1$ first, then since $uv_2$ is uncolored, there are at least $\max\{2\Delta+4,c\}-(2\Delta-1)+1\geq 6$ colors available for $uv_1$.
    Then there are at least $\max\{2\Delta+4,c\}-(2\Delta-1)\geq 5$ colors available for $uv_2$. \\

    \noindent\textbf{Case 3. $d=3$.}

    The number of edges that conflict with $uv_i$ for $i\in\{1,2\}$ is at most $2\cdot 3+2\deg(v_i)+\lfloor 3/2 \rfloor-6=2\deg(v_i)+1\leq 2\Delta+1$.
    
    If we color $uv_1$ first, then since $uv_2$ and $uw_1$ are uncolored, there are at least $\max\bigl\{2\Delta+4,c\bigr\}-\bigl(2\Delta+1\bigr)+2\geq 5$ colors available for $uv_1$.
    Since $uw_1$ is uncolored, there are at least $\max\bigl\{2\Delta+4,c\bigr\}-\bigl(2\Delta+1\bigr)+1\geq 4$ colors available for $uv_2$.

    Finally, the number of edges that conflict with $uw_1$ is at most $2\cdot 3+2\deg(w_1)+\lfloor 3/2\rfloor-6\leq 2\deg(w_1)+1$.
    So, there are at least $\max\{2\Delta+4,c\}-\bigl(2\deg(w_1)+1\bigr)\geq \max\{2\Delta+4,c\}-(2\Delta+1)\geq 3$ colors available for $uw_1$. \\

    \noindent\textbf{Case 4. $d=4$.}

    The number of edges that conflict with $uv_i$ for $i\in\{1,2\}$ is at most $2\cdot 4+2\deg(v_i)+\lfloor 4/2 \rfloor-6=2\deg(v_i)+4\leq 2\Delta+4$.
    
    If we color $uv_1$ first, then since $uv_2$, $uw_1$, and $uw_2$ are uncolored, there are at least $\max\{2\Delta+4,c\}-(2\Delta+4)+3\geq 3$ colors available for $uv_1$.
    Since $uw_1$ and $uw_2$ are uncolored, there are at least $\max\{2\Delta+4,c\}-(2\Delta+4)+2\geq 2$ colors available for $uv_2$.

    The number of edges that conflict with $uw_1$ is at most $2\cdot 4+2\deg(w_1)+\lfloor 4/2\rfloor-6\leq 2\deg(w_1)+4$.
    So there are at least $\max\{2\Delta+4,c\}-(2\deg(w_1)+4)+1$ colors available for $uw_1$.
    Similarly, there are at least $\max\{2\Delta+4,c\}-(2\deg(w_2)+4)$ colors available for $uw_2$.

    In the case of Theorem~\ref{thm:BI}\eqref{it:BI:4a}, \eqref{it:BI:4b}, or \eqref{it:BI:4c}, the number of colors available for $uw_1$ is at least $\max\{2\Delta+4,c\}-(2\deg(w_1)+4)+1\geq \max\{2\Delta+4,c\}-(2\cdot 9+4)+1\geq \max\{2\Delta+4,c\}-21$ and the number of colors available for $uw_2$ is at least $\max\{2\Delta+4,c\}-(2\deg(w_2)+4)\geq \max\{2\Delta+4,c\}-(2\cdot 7+4)\geq \max\{2\Delta+4,c\}-18$.
    Hence $c\geq 22$ suffices. \\

    \noindent\textbf{Case 5. $d=5$.}

    The number of edges that conflict with $uv_i$ for $i\in\{1,2\}$ is at most $2\cdot 5+2\deg(v_i)+\lfloor 5/2 \rfloor-6=2\deg(v_i)+6\leq 2\Delta+6$.
    
    If we color $uv_1$ first, then since $uv_2$, $uw_1$, $uw_2$, and $uw_3$ are uncolored, there are at least $\max\{2\Delta+4,c\}-(2\Delta+6)+4\geq 2$ colors available for $uv_1$.
    Since $uw_1$, $uw_2$, and $uw_3$ are uncolored, there is at least $\max\{2\Delta+4,c\}-(2\Delta+6)+3\geq 1$ color available for $uv_2$.

    The number of edges that conflict with $uw_1$ is at most $2\cdot 5+2\deg(w_1)+\lfloor 5/2\rfloor-6\leq 2\deg(w_1)+6$.
    So there are at least $\max\{2\Delta+4,c\}-(2\deg(w_1)+6)+2$ colors available for $uw_1$, at least $\max\{2\Delta+4,c\}-(2\deg(w_2)+6)+1$ colors available for $uw_2$ and at least $\max\{2\Delta+4,c\}-(2\deg(w_3)+6)$ colors available for $uw_3$. 

    In the case of Theorem~\ref{thm:BI}\eqref{it:BI:5a}, \eqref{it:BI:5b}, or \eqref{it:BI:5c}, the number of colors available for $uw_1$ is at least $\max\{2\Delta+4,c\}-(2\deg(w_1)+6)+2\geq \max\{2\Delta+4,c\}-(2\cdot 7+6)+2\geq \max\{2\Delta+4,c\}-18$, the number of colors available for $uw_2$ is at least $\max\{2\Delta+4,c\}-(2\deg(w_2)+6)+1\geq \max\{2\Delta+4,c\}-(2\cdot 6+6)+1\geq \max\{2\Delta+4,c\}-17$, and the number of colors available for $uw_3$ is at least $\max\{2\Delta+4,c\}-(2\deg(w_3)+6)\geq \max\{2\Delta+4,c\}-(2\cdot 6+6)\geq \max\{2\Delta+4,c\}-18$.
    Hence $c\geq 19$ suffices. \\

    So in every case, having at least $\max\{2\Delta+4,22\}$ colors suffices to B-color the graph $G$. 

\end{proof}

\section{Proofs of Lemmas}

\subsection{Proof of Lemma~\ref{lem:main:outerplanar}} \label{sec:main:outerplanar}
	An edge that conflicts with $xy$ must either be incident to $x$ (there are $\deg(x)-1$ such edges), incident to $y$ (there are $\deg(y)-1$ such edges) or must be an edge parallel to $uv$ in some copy of a $C_4$. 

    An edge parallel to $xy$ cannot have both endvertices in $N(x)\cap N(y)$, otherwise that would result in a $K_4$, contradicting Lemma~\ref{lem:structure:outerplanar}\eqref{outerplanar:noK4}.

    An edge parallel to $xy$ should have one endvertex adjacent to $x$ and the other adjacent to $y$. 
    If $z_1z_2$ and $z_2z_3$ are edges that are parallel to $xy$ then it is not possible that both of $z_1$ and $z_3$ are neighbors of $x$ or that both of them are neighbors of $y$.
    Otherwise, $z_2$ is adjacent to either $y$ or $x$ respectively and so the graph induced by $\bigl\{x,y,z_1,z_2,z_3\bigr\}$ would have a copy of $K_{2,3}$, contradicting Lemma~\ref{lem:structure:outerplanar}\eqref{outerplanar:noK23}.

    So if $z_1z_2$ and $z_2z_3$ are edges that are parallel to $xy$, then without loss of generality, $z_1\in N(x)\setminus N(y)$, $z_2\in N(x)\cap N(y)$ and $z_3\in N(y)\setminus N(x)$.
    Moreover, for the graph induced by the edges that are parallel to $xy$, the components are copies of the path $P_3$ or the path $P_2$ and by Lemma ~\ref{lem:structure:outerplanar}\eqref{outerplanar:noK23} and \eqref{outerplanar:noK4}, it can be seen that these must be vertex-disjoint and there are at most two such paths. 

    Hence, the maximum number of edges parallel to $xy$ is 4 and if there are $k$ parallel edges, then $\deg(x),\deg(y)\geq k+1$. 
    Thus, the total number of conflict edges is at most $d(x)+d(y)-2+\min\bigl\{\deg(x)-1,\deg(y)-1,4\bigr\} = \deg(x)+\deg(y)+\min\bigl\{\deg(x),\deg(y),5\bigr\}-3$.

    If there is a degree-$2$ common neighbor of $x$ and $y$, then the maximum number of edges parallel to $xy$ is 2 and if there are $k$ parallel edges, then $d(x),d(y)\geq k+2$. 
    Thus, if $x$ and $y$ have a degree-$2$ common neighbor, then the total number of conflict edges is at most $\deg(x)+\deg(y)-2+\min\bigl\{\deg(x)-2,\deg(y)-2,2\bigr\} = \deg(x)+\deg(y)+\min\bigl\{\deg(x),\deg(y),4\bigr\}-4$. 
    
    If $xy$ is not in a $C_4$, then the edges that conflict with $xy$ are only the ones incident with either $x$ or $y$, and there are at most $\bigl(\deg(x)-1\bigr)+\bigl(\deg(y)-1\bigr)$ of such edges. \hfill$\Box$

\subsubsection{Proof of Lemma~\ref{lem:main:planar}} \label{sec:main:planar}
For adjacent vertices $u$ and $v$, let
\begin{itemize}
    \item $W_{uv}=N(u)\cap N(v)$,
    \item $W_u\;\,=N(u)-\bigl(N(v)\cup\{v\}\bigr)$, and
    \item $W_v\;\,=N(v)-\bigl(N(u)\cup\{u\}\bigr)$.
\end{itemize}

Let $G'$ be the graph consisting of edges not incident to $u$ or $v$ that are conflict edges with $uv$.
That is, these are edges with one endpoint in $W_u\cup W_{uv}$ and the other endpoint in $W_v\cup W_{uv}$.
In particular, there are no edges in $G'$ with both endvertices in $W_u$ or with both endvertices in $W_v$.

Let $G''$ be the subgraph of $G'$ induced by $W_u\cup W_v$. Note that every edge in $G''$ has one endpoint in $W_u$ and the other endpoint in $W_v$.

We will state several facts about $G'$ and $G''$ that will give a very specific structure to these graphs and will allow us to count the conflict edges.

Fact~\ref{fact:nocycles} establishes that there are very few cycles in $G'$.

\begin{fact}
    There is no cycle of length at least $5$ in $G'$. Furthermore, every $4$-cycle alternates between sets as follows: $W_u,W_{uv},W_v,W_{uv}$ and every $3$-cycle contains either one or two vertices from $W_{uv}$.

    In particular, $G''$ is acyclic.
    \label{fact:nocycles}
\end{fact}

\begin{proof}[Proof of Fact~\ref{fact:nocycles}]
    Suppose that $G'$ has a cycle such that there exist four vertices in cyclic order $x_1\in W_u\cup W_{uv}$, $y_1\in W_v\cup W_{uv}$, $x_2\in W_u\cup W_{uv}$, $y_2\in W_v\cup W_{uv}$.

    This would produce a topological $K_{3,3}$ with one set being $\{x_1,x_2,v\}$ and the other being $\{y_1,y_2,u\}$. Hence, these vertices $x_1,y_1,x_2,y_2$ cannot exist.

    Consider any cycle $v_1,v_2,\ldots,v_{\ell}$ of length $\ell\geq 4$.
    If each of the vertices is in $W_{uv}$, then one can find $x_1,y_1,x_2,y_2$ by choosing them to be consecutive.
    So, without loss of generality, let $v_3\in W_u$.
    If $v_1\not\in W_v$, then since there is no edge in $G'$ with both endvertices in $W_u$ or with both endvertices in $W_v$ we can set $x_1=v_1$, $y_1=v_2$, $x_2=v_3$, and $y_2=v_4$, a contradiction.
    So $v_1\in W_v$.

    If $\ell=4$, then it must be the case that $v_1\in W_v$, $v_2,v_4\in W_{uv}$ and $v_3\in W_u$.

    If $\ell\geq 5$ and $v_5\not\in W_v$, then we can set $x_1=v_5$, $y_1=v_4$, $x_2=v_3$, and $y_2=v_2$, a contradiction.

    If $\ell\geq 5$ and $v_5\in W_v$, then since $v_5v_1$ cannot be an edge, it is the case that $\ell\geq 6$ and we can set $x_1=v_6$, $y_1=v_5$, $x_2=v_3$, and $y_2=v_1$, a contradiction.

    As far as 3-cycles, since there are no edges with both endvertices in $W_u$ or with both endvertices in $W_v$, then any 3-cycle must contain at least one vertex from $W_{uv}$.
    But if $W_{uv}$ contains a 3-cycle, then adding $u$ and $v$ produces a copy of $K_5$, a contradiction.

    Since every cycle in $G'$ must contain a member of $W_{uv}$, the subgraph $G''$ is acyclic.
\end{proof}

Fact~\ref{fact:degtwo} establishes that no vertex has more than two neighbors in $W_{uv}$, which will be important for our expanded definition of a component in $G''$.

\begin{fact}
    Every vertex $x\in V(G')$ has at most two neighbors in $W_{uv}$.
    \label{fact:degtwo}
\end{fact}

\begin{proof}[Proof of Fact~\ref{fact:degtwo}]
    If $x\in V(G')$ has three distinct neighbors $z_1,z_2,z_3\in W_{uv}$, then there is a copy of $K_{3,3}$ with one set being $\{x,u,v\}$ and the other being $\{z_1,z_2,z_3\}$, a contradiction.
\end{proof}

As a result of Fact~\ref{fact:nocycles}, each connected component of $G''$ must be a (possibly degenerate) tree. In addition we will define special components.

\begin{definition}
    The vertex set $V(G'')=W_u\cup W_v$ is partitioned as follows:
    \begin{itemize}
        \item Let every nondegenerate connected component be a member of the partition and call it an \emph{ordinary nondegenerate component}.
        \item If $x\in W_u$ and $y\in W_v$ have two common neighbors in $W_{uv}$, make $\{x,y\}$ a member of the partition and call it a \emph{special component}.
        \item Let every $v\in V(G'')$ that is not in either an ordinary nondegenerate component or in a special component form an \emph{ordinary degenerate component} $\{v\}$.
    \end{itemize}
    Together, these are simply called the \emph{components} of $G''$.
\end{definition}

\begin{remark}
    A special component $\{x,y\}$ may either be a connected component that is a single edge in $G''$ or a pair of vertices that are isolated in $G''$. Taken together with their common neighbors in $W_{uv}$, say $w_1$ and $w_2$, those four vertices induce at least four edges and possibly five with $xy$ or $w_1w_2$. However, both edges cannot be present because, together with $u$ and $v$, this would form a $K_6-2K_2$ which has both a copy of $K_{3,3}$ and a topological copy of $K_5$. 
\end{remark}

It is necessary to establish that components indeed form a partition of $V\bigl(G''\bigr)$ so that we can analyze each component individually. For any component $C$, let $\partial C$ denote the neighbors of $C$ in $W_{uv}$.
That is, $\partial C=\bigl\{z\in W_{uv} : \exists x\in C\subseteq W_u\cup W_v,\; xz\in E(G)\bigr\}$.

\begin{fact}
    The components of $G''$ (including the special components) partition $V\bigl(G''\bigr)=W_u\cup W_v$.
    \label{fact:partition}
\end{fact}

\begin{proof}[Proof of Fact~\ref{fact:partition}]
    It suffices to show that a special component has an empty intersection with every other special component and with every ordinary nondegenerate component.

    If two special components $C_1,C_2$ have a nonempty intersection, then $\partial C_1=\partial C_2$, otherwise some vertex would have three neighbors in $W_{uv}$, contradicting Fact~\ref{fact:degtwo}.
    So, without loss of generality, there is an $x\in W_u$ and distinct $y_1,y_2\in W_v$ such that $x$, $y_1$, and $y_2$ have distinct common neighbors $z_1,z_2\in W_{uv}$.
    This produces a copy of $K_{3,3}$ with one set being $\bigl\{y_1,y_2,u\bigr\}$ and the other being $\bigl\{z_1,z_2,v\bigr\}$.
    See Figure~\ref{fig:partition:specials}.

    If a special component has a nonempty intersection with an ordinary nondegenerate component, then without loss of generality, there is an $x\in W_u$ and distinct $y,y'\in W_v$ such that $x$ and $y$ have common neighbors $z_1,z_2\in W_{uv}$ and there is an edge $xy'$.
    This produces a topological copy of $K_{3,3}$ with one set being $\bigl\{x,y,u\bigr\}$ and the other being $\bigl\{z_1,z_2,v\bigr\}$ (the path $xy'v$ connects $x$ to $v$).
    See Figure~\ref{fig:partition:specconn}.
\end{proof}

\begin{figure}
    \centering
    \begin{subfigure}[t]{0.48\textwidth}
        \centering
        \begin{tikzpicture}[scale=1]
            \ovaldiagram;
            \coordinate (x) at (Wu);
            \coordinate (y2) at ($(Wv)+(-0.3,0.3)$);
            \coordinate (y1) at ($(Wv)+(0.3,-0.5)$);
            \coordinate (z1) at ($(Wuv)+(-0.5,0.2)$);
            \coordinate (z2) at ($(Wuv)+(0.5,0.2)$);
            \begin{pgfonlayer}{forefore}
                \draw (x) node[vtx,label={above:$x$}] {};
                \draw (y1) node[vtx,label={above:$y_1$}] {};
                \draw (y2) node[vtx,label={above:$y_2$}] {};
                \draw (z1) node[vtx,label={below:$z_1$}] {};
                \draw (z2) node[vtx,label={below:$z_2$}] {};
            \end{pgfonlayer}
            \begin{pgfonlayer}{fore}
                \foreach\i in {1,2}{%
                    \draw[edg] (x) -- (z\i);
                    \foreach\j in {1,2}{%
                        \draw[edg] (y\j) -- (z\i);
                    }
                }
            \end{pgfonlayer}
        \end{tikzpicture}
	\caption{Diagram of Fact~\ref{fact:partition} when two special components have a nonempty intersection.} \label{fig:partition:specials}
    \end{subfigure} \hfill
    \begin{subfigure}[t]{0.48\textwidth}
        \centering
        \begin{tikzpicture}[scale=1]
            \ovaldiagram;
            \coordinate (x) at (Wu);
            \coordinate (yp) at ($(Wv)+(-0.3,0.3)$);
            \coordinate (y) at ($(Wv)+(0.3,-0.5)$);
            \coordinate (z1) at ($(Wuv)+(-0.5,0.2)$);
            \coordinate (z2) at ($(Wuv)+(0.5,0.2)$);
            \begin{pgfonlayer}{forefore}
                \draw (x) node[vtx,label={above:$x$}] {};
                \draw (y) node[vtx,label={above:$y$}] {};
                \draw (yp) node[vtx,label={above:$y'$}] {};
                \draw (z1) node[vtx,label={below:$z_1$}] {};
                \draw (z2) node[vtx,label={below:$z_2$}] {};
            \end{pgfonlayer}
            \begin{pgfonlayer}{fore}
                \foreach\i in {1,2}{%
                    \draw[edg] (x) -- (z\i);
                    \draw[edg] (y) -- (z\i);
                }
                \draw[edg] (x) -- (yp);
            \end{pgfonlayer}
        \end{tikzpicture}
        \caption{Diagram of Fact~\ref{fact:partition} when one special component has a nonempty intersection with a nontrivial connected component.} \label{fig:partition:specconn}
    \end{subfigure}
    \caption{Figures for Fact~\ref{fact:partition}}
\end{figure}

Now we must analyze the components of $G''$. Define the following:
\begin{itemize}
    \item   $\C$ is the set of components of $G''$.
    \item   $\C_s\subseteq \C$ is the set of special components of $G''$.
    \item   $\C_i\subseteq \C$ is the set of ordinary (i.e., not in $\C_s$) components $C$ of $G''$ for which $|\partial C|=i$.
\end{itemize}

\begin{fact}
    For every component $C\in\C$, it is the case that $|\partial C|\leq 2$. In particular, $\C=\C_s\cup\C_2\cup\C_1\cup\C_0$.
    \label{fact:bordernbrs}
\end{fact}

\begin{proof}[Proof of Fact~\ref{fact:bordernbrs}]
    By Fact~\ref{fact:degtwo}, no vertex has more than two neighbors in $W_{uv}$.
    Consequently, $|\partial C|=2$ for all $C\in\C_s$.

    Otherwise let $C\in\C-\C_s$, and suppose that there are three (not necessarily distinct) vertices in $V(C)$, called $x_1,x_2,x_3\in W_u\cup W_v$ and three distinct vertices $z_1,z_2,z_3\in W_{uv}$ such that $x_iz_i\in E(G)$ for $i=1,2,3$.

    We leave it to the reader to verify the following basic graph theory fact:
    \begin{fact}
        If $x_1$, $x_2$, and $x_3$ are distinct vertices in a connected graph, then either there is a path that contains all three vertices or there is a topological $K_{1,3}$ whose leaves are $x_1$, $x_2$, and $x_3$. \label{fact:topK13}
    \end{fact}

    If $\bigl\{x_1,x_2,x_3\bigr\}$ is a set of size at most 2, then since $C$ is connected, it is trivial to find a topological $K_{1,3}$ with leaves being $\bigl\{z_1,z_2,z_3\bigr\}$ and all other vertices being in $V(C)$.
    (The center of that $K_{1,3}$ is the member of $\bigl\{x_1,x_2,x_3\bigr\}$ with at least two neighbors in $\bigl\{z_1,z_2,z_3\bigr\}$.)
    Otherwise, $\bigl\{x_1,x_2,x_3\bigr\}$ is a set of size 3 and we can apply Fact~\ref{fact:topK13} to find a topological $K_{1,3}$ in $C$ with leaves being $\bigl\{x_1,x_2,x_3\bigr\}$. 
    In the case of a path, it is a degenerate $K_{1,3}$ in which the center is identified with the leaf that is not an endpoint of the path. 
    Then use the matching $\{x_1z_1,x_2z_2,x_3z_3\}$ to extend it to a topological $K_{1,3}$ with leaves being $\bigl\{z_1,z_2,z_3\bigr\}$.

    If the center of the $K_{1,3}$ is denoted by $w$, then there is a topological copy of $K_{3,3}$ with one set being $\bigl\{z_1,z_2,z_3\bigr\}$ and the other set being $\bigl\{w,u,v\bigr\}$.
    This is a contradiction to the initial assumption. Thus $|\partial C|\leq 2$ for all $C\in\C-\C_s$ as well.
    This concludes the proof of Fact~\ref{fact:bordernbrs}.
\end{proof}

Now we want to count the number of edges.
For any subset of vertices $W$ of $G'$, denote $e[W]=\bigl|E\bigl(G'[W]\bigr)\bigr|$.
Recall that we ignore any edges with both endpoints in $W_u$ or with both endpoints in $W_v$.
Also, for brevity, denote $\|C\|=|V(C)|$, for any $C\in\C$.

\begin{fact}
    Let $C\in\C$.
    \begin{itemize}
        \item   If $C\in\C_0$, then $e[C\cup\partial C]=\|C\|-1$
        \item   For $C\in\C_1$,
                \begin{itemize}
                    \item   if $\|C\|=1$, then $e[C\cup\partial C]=1=\|C\|$;
                    \item   if $\|C\|\geq 2$, then $e[C\cup\partial C]\leq\|C\|+1$.
                \end{itemize}
        \item   For $C\in\C_2$,
                \begin{itemize}
                    \item   if $\|C\|=1$, then $e[C\cup\partial C]\leq 3=\|C\|+2$;
                    \item   if $\|C\|=2$, then $e[C\cup\partial C]\leq 4=\|C\|+2$.
                    \item   if $\|C\|\geq 3$, then $e[C\cup\partial C]\leq\|C\|+3$.
                \end{itemize}
        \item   If $C\in\C_s$, then $e[C\cup\partial C]\leq 5=\|C\|+3$.
    \end{itemize}
    \label{fact:shadowedges}
\end{fact}

\begin{proof}[Proof of Fact~\ref{fact:shadowedges}]  We start by making the following observation.
    \begin{observation}
    	No vertex in $W_{uv}$ has two neighbors in $W_u$ or two neighbors in $W_v$ that are in the same component. \label{obs:no2nbrs}
    \end{observation}
    To see why the observation is true, without loss of generality, suppose that $z\in W_{uv}$ has two neighbors $y_1,y_2\in W_v\cap V(C)$ for some component $C$.
    Note that $C\not\in\C_s$ by definition. 
    Since $y_1$ and $y_2$ are in the same connected component, there is a path from $y_1$ to $y_2$ in $G''$ that includes at least one vertex $x\in W_u$.
    This forms a topological copy of $K_{3,3}$ with one set being $\bigl\{y_1,y_2,u\bigr\}$ and the other set being $\bigl\{x,z,v\bigr\}$, a contradiction.
    See Figure~\ref{fig:no2nbrs}.
    This establishes Observation~\ref{obs:no2nbrs}. \\

    With Observation~\ref{obs:no2nbrs} in mind, most of the items become trivial.

    \begin{itemize}
        \item   If $C\in\C_0$, then all of the edges are in $C$ and since $C$ is a tree, the number of edges is precisely $\|C\|-1$.
        \item   If $C\in\C_1$, then there is exactly one $z\in W_{uv}$ that has neighbors in $C$.
        If $\|C\|=1$, then this is a total of 1 edge, which must exist if $C\in\C_1$.
        If $\|C\|\geq 2$, then Observation~\ref{obs:no2nbrs} gives that there are at most two edges between $C$ and $\partial C$, which gives a total of at most $\bigl(\|C\|-1\bigr)+2=\|C\|+1$ edges.
        \item   If $C\in\C_2$, then there are exactly two $z_1,z_2\in W_{uv}$ that have neighbors in $C$.
        \begin{itemize}
            \item  If $\|C\|=1$, then there must be two edges between $C$ and $\partial C$.
            Together with a possible edge $z_1z_2$, this gives at most $3=\|C\|+2$ edges.
            \item   If $\|C\|=2$, then, because $C\in\C_2$, the vertices of $C$ must be $x\in W_u$ and $y\in W_v$ and $xy$ must be an edge.
            Furthermore, there are at most three edges between $C$ and $\partial C$, otherwise $C$ would be in $\C_s$.
            Without loss of generality, suppose the edges between $C$ and $\partial C$ are $xz_1$, $xz_2$ and $yz_2$.
            If $z_1z_2$ is an edge, then this gives a topological $K_5$ with vertices $\{u,v,z_1,z_2,x\}$ that uses the path $xyv$, a contradiction.
            See Figure~\ref{fig:shadowedges:22edge}.
            Hence either $z_1z_2$ is an edge and there are at most two edges between $C$ and $\partial C$ or $z_1z_2$ is not an  edge and there are at most three edges between $C$ and $\partial C$.
            In either case, $e[C\cup\partial C]\leq 1+3=\|C\|+2$
            \item   If $\|C\|\geq 3$, observe that $e[C\cup\partial C]\leq\bigl(\|C\|-1\bigr)+5= \|C\|+4$, with equality if and only if $z_1z_2$ is an edge and each of $z_1,z_2$ has exactly one neighbor in $W_u$ and exactly one neighbor in $W_v$.
            We may, without loss of generality, assume that there are distinct $y_1,y_2\in W_v$ and $x\in W_u$ such that $xz_1,y_1z_2,y_2z_1$ are edges between $C$ and $\partial C$.
            By Fact~\ref{fact:topK13}, either $y_1$, $x$ and $y_2$ are in the same path or there exists a $w\in V(C)$ such that $C$ contains a topological copy of $K_{1,3}$ with $w$ as the center and $x,y_1,y_2$ as the leaves.
                \begin{itemize}
                    \item If the order of vertices in the path is $y_1,x,y_2$, then there is a topological $K_{3,3}$ with one set being $\bigl\{y_1,y_2,u\bigr\}$ and the other set being $\bigl\{x,z_1,v\bigr\}$. 
                    See Figure~\ref{fig:shadowedges:23path}.
                    \item If the order of vertices in the path is $x,y_1,y_2$ or $y_1,y_2,x$, let $x'$ be a vertex in $W_u$ that is along the path from $y_1$ to $y_2$, then there is a topological $K_{3,3}$ with one set being $\bigl\{y_1,y_2,u\bigr\}$ and the other set being $\bigl\{x,x',v\bigr\}$. 
                    This is because there are paths with vertices in the following order $x,z_1,y_1$ and $x,z_1,z_2,y_2$, as well as edges between the rest of the pairs. 
                    \item If there is a star centered at $w$, then there is a topological $K_{3,3}$ with one set being $\bigl\{x,y_1,y_2\bigr\}$ and the other set being $\bigl\{w,z_1,v\bigr\}$. 
                    This is because there are paths with endpoints $w,x$, and $w,y_1$ and $w,y_2$, as well as edges between the rest of the pairs. 
                \end{itemize}
            Since each case is a contradiction, $e[C\cup\partial C]\leq \|C\|+3$, as claimed.
        \end{itemize}
        \item   If $C\in\C_s$, then $C\cup\partial C$ spans $4$ vertices.
        Let $C=\{x,y\}$ with $x\in W_u$ and $y\in W_v$.
        If $e[C\cup\partial C]=6$ then the graph induced by $C\cup\partial C\cup\{u,v\}$ induces a copy of $K_6-2K_2$.
        That is, it is a complete graph on $6$ vertices with two nonadjacent edges missing (specifically $uy$ and $xv$).
        See Figure~\ref{fig:shadowedges:Cs}.
        Such a graph has a copy of $K_{3,3}$ (and a topological copy of $K_5$), a contradiction.
        Thus, there are at most $5=\|C\|+3$ edges induced by $C\cup\partial C$.
    \end{itemize}
\end{proof}

\begin{figure}
    \centering
    \begin{subfigure}[t]{0.48\textwidth}
        \centering
        \begin{tikzpicture}[scale=1]
            \ovaldiagram;
            \coordinate (x) at (Wu);
            \coordinate (y2) at ($(Wv)+(-0.3,0.3)$);
            \coordinate (y1) at ($(Wv)+(0.3,-0.5)$);
            \coordinate (z) at ($(Wuv)+(0,0.2)$);
            \begin{pgfonlayer}{forefore}
                \draw (x) node[vtx,label={above:$x$}] {};
                \draw (y1) node[vtx,label={above:$y_1$}] {};
                \draw (y2) node[vtx,label={above:$y_2$}] {};
                \draw (z) node[vtx,label={below:$z$}] {};
            \end{pgfonlayer}
            \begin{pgfonlayer}{fore}
                \foreach\i in {1,2}{%
                    \draw[edg] (z) -- (y\i);
                    \draw[pth] (x) -- (y\i);
                }
            \end{pgfonlayer}
        \end{tikzpicture}
	   \caption{Diagram of Observation~\ref{obs:no2nbrs}. The wavy line represents a path that begins at $y_1$, ends at $y_2$, contains $x$ and alternates between $W_v$ and $W_u$.}
        \label{fig:no2nbrs}
    \end{subfigure} \hfill
    \begin{subfigure}[t]{0.48\textwidth}
        \centering
        \begin{tikzpicture}[scale=1]
            \ovaldiagram;
            \coordinate (x) at (Wu);
            \coordinate (y) at (Wv);
            \coordinate (z1) at ($(Wuv)+(-0.5,0.2)$);
            \coordinate (z2) at ($(Wuv)+(0.5,0.2)$);
            \begin{pgfonlayer}{forefore}
                \draw (x) node[vtx,label={above:$x$}] {};
                \draw (y) node[vtx,label={above:$y$}] {};
                \draw (z1) node[vtx,label={below:$z_1$}] {};
                \draw (z2) node[vtx,label={below:$z_2$}] {};
            \end{pgfonlayer}
            \begin{pgfonlayer}{fore}
                \foreach\i in {1,2}{%
                    \draw[edg] (x) -- (z\i);
                }
                \draw[edg] (x) -- (y);
                \draw[edg] (y) -- (z2);
                \draw[edg] (z1) -- (z2);
            \end{pgfonlayer}
        \end{tikzpicture}
	   \caption{Diagram of Fact~\ref{fact:shadowedges} in the case where $C\in\C_2$, $\|C\|=2$ and there is an edge $z_1z_2$.}
        \label{fig:shadowedges:22edge}
    \end{subfigure}
    \caption{Figures for Fact~\ref{fact:shadowedges}}
\end{figure}

\begin{figure}
    \centering
    \begin{subfigure}[t]{0.48\textwidth}
        \centering
        \begin{tikzpicture}[scale=1]
            \ovaldiagram;
            \coordinate (x) at (Wu);
            \coordinate (y2) at ($(Wv)+(-0.3,0.3)$);
            \coordinate (y1) at ($(Wv)+(0.3,-0.5)$);
            \coordinate (z1) at ($(Wuv)+(-0.5,0.2)$);
            \coordinate (z2) at ($(Wuv)+(0.5,0.2)$);
            \begin{pgfonlayer}{forefore}
                \draw (x) node[vtx,label={above:$x$}] {};
                \draw (y1) node[vtx,label={above:$y_1$}] {};
                \draw (y2) node[vtx,label={above:$y_2$}] {};
                \draw (z1) node[vtx,label={below:$z_1$}] {};
                \draw (z2) node[vtx,label={below:$z_2$}] {};
            \end{pgfonlayer}
            \begin{pgfonlayer}{fore}
                \foreach\i in {1,2}{%
                    \draw[pth] (x) -- (y\i);
                }
                \draw[edg] (x) -- (z1);
                \draw[edg] (y2) -- (z1);
                \draw[edg] (y1) -- (z2);
                \draw[edg] (z1) -- (z2);
            \end{pgfonlayer}
        \end{tikzpicture}
	   \caption{Diagram of Fact~\ref{fact:shadowedges} in the case where $C\in\C_2$, $\|C\|\geq 3$ and $y_1$, $x$, and $y_2$ are in the same path. The wavy line represents a path that begins at $y_1$, ends at $y_2$, contains $x$ and alternates between $W_v$ and $W_u$.}
        \label{fig:shadowedges:23path}
    \end{subfigure} \hfill
    \begin{subfigure}[t]{0.48\textwidth}
        \centering
        \begin{tikzpicture}[scale=1]
            \ovaldiagram;
            \coordinate (x) at (Wu);
            \coordinate (y) at (Wv);
            \coordinate (z1) at ($(Wuv)+(-0.5,0.2)$);
            \coordinate (z2) at ($(Wuv)+(0.5,0.2)$);
            \begin{pgfonlayer}{forefore}
                \draw (x) node[vtx,label={above:$x$}] {};
                \draw (y) node[vtx,label={above:$y$}] {};
                \draw (z1) node[vtx,label={below:$z_1$}] {};
                \draw (z2) node[vtx,label={below:$z_2$}] {};
            \end{pgfonlayer}
            \begin{pgfonlayer}{fore}
                \foreach\i in {1,2}{%
                    \draw[edg] (x) -- (z\i);
                    \draw[edg] (y) -- (z\i);
                }
                \
                \draw[edg] (x) -- (y);
                \draw[edg] (z1) -- (z2);
            \end{pgfonlayer}
        \end{tikzpicture}
	   \caption{Diagram of Fact~\ref{fact:shadowedges} in the case where $C\in\C_s$ and $e[C\cup\partial C]=6$.}
        \label{fig:shadowedges:Cs}
    \end{subfigure}
    \caption{Figures for Fact~\ref{fact:shadowedges}}
\end{figure}

Now we are ready to complete the proof of Lemma~\ref{lem:main:planar}.
First, we will require a bit more notation:
\begin{itemize}
    \item Define a partition $\C_1=\C_1'(u)\cup\C_1'(v)\cup\C_1''$; such that
    \begin{itemize}
        \item for $\alpha\in\{u,v\}$, let $\C_1'(\alpha)$ denote the set of $C\in\C_1$ with $\|C\|=1$ such that its vertex is in $W_{\alpha}$,
        \item let $\C_1''$ denote the set of $C\in\C_1$ with $\|C\|\geq 2$.
    \end{itemize}
    \item Define a partition $\C_2=\C_2'(u)\cup\C_2'(v)\cup\C_2''\cup\C_2'''$; such that
    \begin{itemize}
        \item for $\alpha\in\{u,v\}$, let $\C_2'(\alpha)$ denote the set of $C\in\C_2$ with $\|C\|=1$ such that its vertex is in $W_{\alpha}$,
        \item let $\C_2''$ denote the set of $C\in\C_2$ with $\|C\|=2$, and
        \item let $\C_2'''$ denote the set of $C\in\C_2$ with $\|C\|\geq 3$.
    \end{itemize}
    \item Let $\K_2$ denote the set of edges $z_1z_2$ in $G\bigl[W_{uv}\bigr]$ such that there is no $C\in\C_2\cup \C_s$ for which  $\partial C=\{z_1,z_2\}$.  These are the only edges not accounted for by considering $C\cup\partial C$ for some $C\in\C$.
\end{itemize}

By Fact~\ref{fact:shadowedges}, the total number of edges that conflict with $uv$ is
\begin{align}
    \lefteqn{~\hspace{-24pt}
    \bigl(\deg(u)-1\bigr) + \bigl(\deg(v)-1\bigr) + |\K_2| + \sum_{C\in\C} e[C\cup\partial C]} \nonumber \\
    &\leq \bigl(|W_u|+|W_{uv}|\bigr) + \bigl(|W_v|+|W_{uv}|\bigr) + |\K_2| + \sum_{C\in\C} \|C\| \nonumber \\
    &\hspace{12pt} -  |\C_0|  + |\C_1''| + 2|\C_2'(u)| + 2|\C_2'(v)| + 2|\C_2''| + 3|\C_2'''| + 3|\C_s| \nonumber \\
    &= 2|W_u| + 2|W_v| + 2|W_{uv}| - |\C_0| \nonumber \\
    &\hspace{12pt} + |\K_2| + |\C_1''| + 2|\C_2'(u)| + 2|\C_2'(v)| + 2|\C_2''| + 3|\C_2'''| + 3|\C_s| . \label{eq:edgecount:planar}
\end{align}

Now we consider our two different conclusions, the general bound and the bound in which $\deg(u)$ and the degrees of its neighbors other than $v$ are small. \\

If $|W_{uv}|=0$ and $\min\{\deg(u),\deg(v)\}\geq 2$, then the number of conflict edges is at most
\begin{align*}
    2|W_u| + 2|W_v| - |\C_0| \leq 2\deg(u) + 2\deg(v) - 5 .
\end{align*}

Otherwise, we need an inequality to bound the number of members of $\C_2\cup\C_s$.
\begin{fact} Let $|W_{uv}|\geq 1$.
\begin{align}
    |\C_1'(u)| + |\C_1''| + |\C_2'(u)| + |\C_2''| + |\C_2'''| + |\C_s|     &\leq   |W_u|, \label{eq:consu:planar} \\
    |\C_1'(v)| + |\C_1''| + |\C_2'(v)| + |\C_2''| + |\C_2'''| + |\C_s| &\leq   |W_v|, \label{eq:consv:planar} \\
    2|\K_2| + |\C_1''| + 2|\C_2| + 2|\C_s|   &\leq   2|W_{uv}|, \label{eq:consuv:planar} \\
    |\K_2| + |\C_2| + |\C_s|     &\leq   |W_{uv}|-1. \label{eq:consuvm1:planar}
\end{align}
    \label{fact:Wuvedges}
\end{fact}

\begin{proof}[Proof of Fact~\ref{fact:Wuvedges}]
    The case of inequality \eqref{eq:consu:planar} simply follows from the fact that the components of the type mentioned on the left-hand side of~\eqref{eq:consu:planar} partition a subset of $W_u$. 
    The case of inequality \eqref{eq:consv:planar} follows similarly when applied to $W_v$.

	For inequalities \eqref{eq:consuv:planar} and \eqref{eq:consuvm1:planar}, we create an auxiliary looped graph $\Gamma$ such that $V(\Gamma)=W_{uv}$ and edges as follows: 
    (a) For any distinct $z_1,z_2\in V(\Gamma)$, there is an edge $z_1z_2\in E(\Gamma)$ if and only if either there is a $C\in\C_2\cup\C_s$ such that $\partial C=\{z_1,z_2\}$ or $z_1z_2\in\K_2$. 
    (b) The vertex $z$ has a loop for every $C\in\C_1$ such that $\partial C=\{z\}$.

	The case of~\eqref{eq:consuv:planar} will follow from showing that the maximum degree of $\Gamma$ is at most 2, where loops count once. Let $z$ have degree 3 and for each edge incident to $z$, including loops, denote the vertices $x_i$ and $y_i$ as follows:
	\begin{itemize}
		\item If there exists a $z_i\in W_{uv}-\{z\}$ such that $zz_i$ is an edge in $\K_2$, then let $x_i=z_i$ and $y_i=z_i$ also.
		\item If $C\in\C_1''$ satisfies $\partial C=\{z\}$, then let $x_i$ is a neighbor of $z$ in $\bigl(W_u\cup W_v\bigr)\cap C$. 
        If $x_i\in W_u$ then let $y_i$ be a neighbor of $x_i$ in $W_v$. If $x_i\in W_v$ then let $y_i$ be a neighbor of $x_i$ in $W_u$. 
		\item If $C\in\C_2$ satisfies $\partial C=\{z,z_i\}$, then there is a path from the neighbor of $z$ in $C$ (let this neighbor be $x_i$) to $z_i$ (let $y_i=z_i$).
		\item If $C\in\C_s$ satisfies $\partial C=\{z,z_i\}$, then $C\cup\partial C$ has a path $zx_iz_i$, where $x_i\in W_u$. Let $y_i=z_i$.
	\end{itemize}

    This produces a topological copy of $K_{3,3}$ with sets $\{z,u,v\}$ and $\{x_1,x_2,x_3\}$. For $i=1,2,3$, there are edges $zx_i$. 
    In addition, either there is an edge $ux_i$ and a $vx_i$ path with $y_i$ as an intermediate vertex or an edge $vx_i$ and a $ux_i$ path with $y_i$ as an intermediate vertex. 
    Each such path is internally disjoint. 
    This produces a topological $K_{3,3}$ and is a contradiction. 
    Therefore, every vertex in $\Gamma$ has degree at most two and because the sum of the degrees is at most $2\bigl|V(\Gamma)\bigr|=2\bigl|W_{uv}\bigr|$, we have established \eqref{eq:consuv:planar}.

    As for inequality \eqref{eq:consuvm1:planar}, we want to show that $\Gamma$ has no cycles (other than the loops).  
    Indeed, for every edge $z_1z_2\in E(\Gamma)$, there is either an edge $z_1z_2\in E(G)$ or a distinct component $C$ such that there is a path from $z_1$ to $z_2$ such that all of its internal vertices are members of $C$. 
    So if $\Gamma$ has a non-loop cycle, then $G'$ has a cycle with at least three vertices in $W_{uv}$. However, Fact~\ref{fact:nocycles} gives that $G'$ can have no cycles with three vertices in $W_{uv}$.
    
    Thus $\Gamma$ has no non-loop cycles and if $E(\Gamma)$ denotes the non-loop edge set, then $\bigl|E(\Gamma)\bigr|\leq \bigl|V(\Gamma)\bigr|-1=\bigl|W_{uv}\bigr|-1$, establishing \eqref{eq:consuvm1:planar}.
\end{proof}

Now we use Fact~\ref{fact:Wuvedges} to bound the expression in \eqref{eq:edgecount:planar}.
Recall that if $|W_{uv}|=0$ and $\min\bigl\{\deg(u),\deg(v)\bigr\}\geq 2$, then the number of conflict edges is at most $2\deg(u)+2\deg(v)-5$.
Otherwise, there are three bounds we use for the number of conflict edges:
\begin{align*}
    \eqref{eq:edgecount:planar} &\leq    2|W_u| + 2|W_v| + 2|W_{uv}| \\
    &\hspace{36pt} + \min    \left\{
            \begin{array}{l}
                3|W_{uv}|-1,
                \\
                \min\bigl\{|W_u|,|W_v|\bigr\} + 2 \bigl(|W_{uv}| - 1\bigr) ,
                \\
                \bigl(|W_u|+|W_v|\bigr) + \bigl(|W_{uv}| - 1\bigr) 
            \end{array}
                \right\} .
\end{align*}

The first comes from adding $\eqref{eq:consuv:planar}+\eqref{eq:consuvm1:planar}$. The second comes from adding $\min\left\{\eqref{eq:consu:planar},\eqref{eq:consv:planar}\right\}+2\cdot\eqref{eq:consuvm1:planar}$. 
The third comes from adding $\eqref{eq:consu:planar}+\eqref{eq:consv:planar}+\eqref{eq:consuvm1:planar}$.
Finally, the minimum is taken over the right-hand side.

Now we find an upper bound in terms of $\deg(u)=\bigl|W_u\bigr|+\bigl|W_{uv}\bigr|+1$, $\deg(v)=\bigl|W_v\bigr|+\bigl|W_{uv}\bigr|+1$, and $\deg(u,v)=\bigl|W_{uv}\bigr|$. 
\begin{align}
    \eqref{eq:edgecount:planar} &\leq    2\deg(u) + 2\deg(v) - 6 \nonumber \\
    &\hspace{12pt} + \min\left\{
            \begin{array}{l}
                \deg(u,v)+1,\\
                \min\left\{\deg(u),\deg(v)\right\} - \deg(u,v) - 1 ,\\
                \deg(u)+\deg(v) - 3\deg(u,v) - 1 
            \end{array}\right\} .
    \label{eq:edgecount:final}
\end{align}

\newcommand{\m}{\mathfrak{m}}
\newcommand{\M}{\mathfrak{M}}
Let $\m(u,v)=\min\{\deg(u),\deg(v)\}$ and $\M(u,v)=\max\{\deg(u),\deg(v)\}$.
If $2\deg(u,v)\leq\m(u,v)-2$, then
\begin{align*}
    \deg(u,v)+1\leq \tfrac{1}{2}\m(u,v) .
\end{align*}
If $\m(u,v)-1\leq 2\deg(u,v)\leq \M(u,v)-1$, then by substituting the lower bound for $\deg(u,v)$,
\begin{align*}
    \min\left\{\deg(u),\deg(v)\right\} - \deg(u,v) - 1 &\leq \tfrac{1}{2}\m(u,v) -\tfrac{1}{2}.
\end{align*}
If $\M(u,v)\leq 2\deg(u,v)$, then
\begin{align*}
    \deg(u)+\deg(v) - 3\deg(u,v) - 1 &\leq \m(u,v) + \M(u,v) - 3\bigl(\tfrac{1}{2}\M(u,v)\bigr) - 1 \\
    &\leq \tfrac{1}{2}\m(u,v) + \tfrac{1}{2}\bigl(\m(u,v)-\M(u,v)\bigr)-1 \\
    &\leq \tfrac{1}{2}\m(u,v) -1.
\end{align*}

Note that the upper bound of $\frac{1}{2}\m(u,v)$ we obtain for  \eqref{eq:edgecount:final} is consistent with the case of $|W_{uv}|=\deg(u,v)=0$ and $\m(u,v)\geq 2$, which yields an upper bound of $2\deg(u)+2\deg(v)-5$.
Since \eqref{eq:edgecount:planar} must be an integer, the expression gives the desired general bound
\begin{align*}
    \eqref{eq:edgecount:planar} &\leq   2\deg(u) + 2\deg(v) -6  +\left\lfloor\frac{1}{2}\min\left\{\deg(u),\deg(v)\right\}\right\rfloor .
\end{align*} \hfill\qed 

\section{Concluding remarks}

In the proof of Theorem~\ref{thm:outerplanar}, there is only one place which requires $c>7$ and that is the very end of the proof. 
Thus, it may be possible to improve the constant term for outerplanar graphs to $\max\{\Delta,c\}$ for some $c<9$.

We note that a more specific lower bound for $q_B(G)$ is $\max\{\Delta,2\Delta_2\}$, where $\Delta_2$ is the largest $t$ for which $G$ contains a subgraph $K_{2,t}$. 
We call $\Delta_2$ the {\em maximum common degree} in this case.
Is it possible that there is a constant $c$ such that for any planar graph $G$ with maximum degree $\Delta$ and maximum common degree $\Delta_2$, $q_B(G)\leq\max\{\Delta,2\Delta_2,c\}$?

We could reduce the bound for planar graphs in Theorem~\ref{thm:planar} to an expression of the form $\max\{2\Delta+3,c\}$ if we could affirmatively resolve the following question:
\begin{question}
    Does there exist an absolute constant $c$ such that for every planar graph with minimum degree $\delta\geq 3$, there exists one of the following:
    \begin{itemize}
        \item a vertex of degree $3$ with a neighbor of degree at most $c$,
        \item a vertex of degree $4$ with two neighbors, each of degree at most $c$, or
        \item a vertex of degree $5$ with four neighbors, each of degree at most $c$?
    \end{itemize}
\end{question}
We note that it is not possible to replace the second bullet point to ``a vertex of degree $4$ with three neighbors, each of degree at most $c$'' because the double-wheel would contradict that. 
In the double-wheel, there are two vertices of degree $\Delta=n-2$ and $n-2$ vertices of degree $4$ such that each degree-$4$ vertex is adjacent to both degree-$\Delta$ vertices. 

\section{Acknowledgements}

The authors would like to thank Andr\'as Gy\'arf\'as for many valuable conversations and for comments on early versions of the manuscript. 
The first author's research was partially supported by Simons Foundation Collaboration Grant for Mathematicians \#709641 and this research was partially done while this author was on an MTA Distinguished Guest Scientist Fellowship 2023 at the HUN-REN Alfred Renyi Institute of Mathematics and while on a Fulbright Scholarship to the University of Birmingham.
The second and third author's research was supported in part by Nemzeti Kutat\'asi Fejleszt\'esi \'es Innov\'aci\'os Hivatal (NKFIH) Grant No. K132696. The third author's research was also supported in part by NKFIH Grant No. K117879.

We are grateful to anonymous referees, who identified errors in a previous version of the manuscript.

\end{document}